\documentclass[12pt, a4paper]{amsart}
\usepackage{amsmath,amsfonts,amssymb}
\usepackage{setspace}
\usepackage{longtable}
\usepackage[mathscr]{eucal}
\usepackage{amsmath, amsthm}
\usepackage{mathrsfs}
\usepackage{amsbsy}
\usepackage{dsfont}
\usepackage{bbm}
\usepackage{wasysym}
\usepackage{stmaryrd}
\usepackage{url}

\usepackage{color}

\newtheorem{theorem}[subsection]{Theorem}

\newtheorem{corollary}[subsection]{Corollary}

\input xypic
\xyoption{all}
\usepackage[breaklinks]{hyperref}
\newcommand \ZZ {{\mathbb Z}}

\newcommand \PR {{\mathbb P}}

\newcommand \Pic {{\rm {Pic}}}

\newcommand \CH {{\it {CH}}}
\newcommand \wt {\widetilde }

\usepackage{graphicx}
\begin{document}

\title{Universal codimension two cycle on a very general cubic threefold}

\author{ Kalyan Banerjee}
\address{SRM AP, India}
\email{kalyan.ba@srmap.edu.in}
\maketitle

\begin{abstract}
In this paper, we prove that a very general cubic threefold does not admit a universal codimension-two cycle and hence is stably irrational.
\end{abstract}

\section {Introduction}

The rationality of smooth, projective varieties is an important and classical question in algebraic geometry. In the 1970s Clemens and Griffiths \cite{CG} proved that a smooth cubic threefold in 
$\mathbb{P}^4$ is non-rational. The technique to prove this was delicate and elegant. They proved that if the cubic is rational then it's intermediate Jacobian is isomorphic to Jacobian of a smooth, projective curve and then they prove that such phenomena does not occur. From then on, it is an important and unsettled quest about the stable rationality of a cubic threefold. That is, whether for a smooth cubic threefold $X$, $X\times \mathbb{P}^n$ is birational to $\mathbb{P}^r$, for some $n,r$ positive integers. In this paper, we settle this question for a very general cubic threefold $X$. That is, we prove that

\smallskip

\textit{A very general cubic threefold in $\mathbb {P}^4$ is not stably rational.}

\smallskip

The method to show this is to use the integral Chow-theoretic decomposition of the diagonal of the cubic. We show that it is not possible to have an integral Chow-decomposition of the diagonal of the cubic threefold. This is proven by showing that there is no universal codimension 2 cycle on the product $P_X\times X$, where $P_X$ is the Prym variety associated to $X$. This is inspired by the work of Claire Voisin, \cite{Voi1}, \cite{Voi2}, where she proves that for cubic hypersurfaces, the integral Chow theoretic decomposition of diagonal is equivalent to the integral cohomological decomposition of the diagonal and from that it follows that there exists a universal codimension 2 cycle on the product $P_X\times X$ satisfying certain properties. 

We show that if such a codimension-two cycle exists then the Prym variety of $X$, is isogenous to the Jacobian of a smooth, projective curve, which contradicts a result of Naranjo and Pirola, \cite{NP}, that for a very general cubic threefold it is not possible. In this regard, we must also mention that there is an approach by Engel, Fortman and Schreider \cite{EFS} that leads to the same result as in this article but is motivated by tropical geometry and independent of our approach.

Convention: We work over the field of complex numbers.

\section{Survey of known results and definitions}

In this section, we recall the definitions and results to write the proof of the result.

Let $X$ be a smooth, projective cubic threefold in $\PR^4$. Consider the diagonal in $X\times X$ denoted as $\Delta_X$. We say that $\Delta_X$ admits a Chow theoretic decomposition if it can be written as
$$x\times X+Z$$
modulo rational equivalence, 
where $x$ is a closed point on $X$ and $Z$ is supported on $X\times D$, where $D$ is a proper Zariski closed subset in $X$.

We say that $\Delta_X$ admits a cohomological decomposition of the diagonal if the above decomposition occurs in $H^6(X\times X,\ZZ)$.

We recall the notion of universal codimension two cycles on $P_X\times X$, where $X$ is a cubic threefold and $P_X$ the Prym variety of $X$. 

Let $Z\in \CH^2(P_X\times X)$ is a universal codimension two cycle if $Z|_{a\times X}$ is homologous to zero for all $a$ in $P_X$ and we have 
$$a\mapsto AJ(Z_a)=Id\;.$$
That is the composition of $AJ$ with $Z_*$ from $P_X$ to $P_X$ is identity.

There is an important result of Voision \cite{Voi1}, namely,

\begin{theorem}\cite{Voi1}
Let $X$ be a smooth cubic hypersurface. Assume $H^*(X,\ZZ)/H^*(X, \ZZ)_{alg}$ has no 2-torsion (this holds, for example, if $\dim X$ is odd or $\dim X\leq 4$, or $X$ is very general of any dimension). Then X admits a Chow-theoretic decomposition of the diagonal if and only if it admits a cohomological decomposition of the diagonal.
\end{theorem}

Then another important result, which connects the cohomological decomposition with the existence of the universal codimension 2 cycle, is as follows:

\begin{theorem}\cite{Voi}\label{thm1}
Let $X$ be a smooth, rationally connected threefold. Then $X$
admits a cohomological decomposition of the diagonal if and only if the following three
conditions are satisfied:
(i) $H^3(X, \ZZ)$ has no torsion.

(ii) There exists a universal codimension 2 cycle in $X \times P_X.$

(iii) The minimal class $\theta^{g-1}/(g-1)!$ on $P_X$, $\dim P_X = g$, is algebraic, that is, the
class of a 1-cycle in $P_X.$
\end{theorem}

The next result that we use in our paper is the following due to Naranjo and Pirola:

\begin{theorem}\cite{NP}
Let $X$ be a very general cubic threefold then $P_X$ is not isogneous to the Jacobian of a smooth, projective curve $C$.
\end{theorem}

\section{ Stable irrationality of smooth cubic threefolds in $\PR^4$}

In this section, we prove the following.

\begin{theorem}
A very general smooth cubic threefold embedded in $\PR^4$ is not stably rational.
\end{theorem}
\begin{proof}
Suppose that for a very general smooth cubic threefold $X$ in $\PR^4$ it is stably rational. Then the diagonal of $X$ admits an integral Chow theoretic decomposition and  we have 
$$\Delta_X=x\times X+Z$$
modulo rational equivalence,
where $Z$ is supported on $X\times D$, $D$ is a proper algebraic subset of $X$. Now, this correspondence $\Delta_X$ acts as the identity automorphism on $A_1(X)$ (the group of algebraically trivial one cycles on $X$ modulo rational equivalence). By the integral decomposition of the diagonal we have 
$$\Delta_{X*}=(x\times X)_*+Z_*$$
Therefore by the Chow moving lemma we have $$Id = Z_ *$$
This is because the action of $x\times X$ on $A_1(X)$ is zero (we can move the support of the one cycle away from $x$ modulo rationally equivalence).
Therefore, we need to examine the action of $Z_*$ on $A_1(X)$. By the very definition of $Z$ we have that $im(Z_*)$ is supported on $A_1(D)$, one cycles on the proper Zariski closed subset $D$. Now $D$ is a reducible surface (possibly singular) at finitely many points. Then we have $D_i$'s as its irreducible components and $\wt{D_i}$'s are its desingularizations. It follows that $\wt{D}=\cup_i \wt{D_i}$ inside the blow-up $\wt{X}$ of $X$ at these finitely many points has the following property
$A_1(\wt{D})\to A_1(D)$ is surjective. Now $A_1(\wt{D})$ is dominated by the direct sum of $\Pic^0(\wt{D_i})$. Thus we have 
$$\oplus \Pic^0(\wt{D_i})\to A_1(D)\to A_1(X)$$
is surjective. Now, identifying $A_1(X)$ with $P_X$, the Prym variety of $X$ we have $\oplus \Pic^0(\wt{D_i})\to P_X$ an onto map. But $im(Pic^0(\wt{D_i}))$ in $P_X$ is an abelian subvariety of $P_X$. Since $P_X$ is simple, we see that the image is either zero or all of $P_X$. All the images of all $\Pic^0(\wt{D_i})$ cannot be zero as, then $P_X$ is zero, which is not possible by any chance.
Therefore, there exists $D_i$ such that $\wt{D_i}$ surjects onto $P_X$, for a single component $D_i$ of $D$. So we have $\Pic^0(\wt{D_i})$ surjects onto $A_1(X)$, which is isomorphic to $P_X$. Since we have the Chow theoretic decomposition of the diagonal, we have the cohomological decomposition of the diagonal and hence there exists a universal codimension-two cycle by \ref{thm1}. Consider the universal codimension two cycle $Z$ on $P_X\times X$. Hence we have $A_1(D_i)$ to $P_X$ is surjective. Taking $A_1(\wt{D_i})$ as $\Pic^0(\wt{D_i})$, we have 
$$\Pic^0(\wt{D_i})\to P_X$$
surjective. Hence, the dual map 
$$P_X^{\vee}\to Alb(\wt{D_i})$$
is injective. We have $P_X$ is isogenous to $ P_X^{\vee}$ by a principal isogeny.
Also we have 
$$P_X^{\vee}\to Alb(\wt{D_i})\to Pic^0(\wt{D_i})$$
is injective, composition with the principal isogeny $Alb(\wt{D_i})\to Pic^0(\wt{D_i})$. Now compose it with $\Pic^0(\wt{D_i})\to P_X$, then we have a map of abelian varieties from $$P_X^{\vee}\to P_X$$
Both abelian varieties are simple and hence the kernel is finite or it is a zero map.

Suppose that the map from $P_X^{\vee}$ to $P_X$ is zero. Then we have the pushforward homomorphism from $\Pic^0(\wt{D_i})\to A_1(X)$  containing an abelian variety that is isogenous to $P_X$. Now, the universal codimension two cycle on $P_X\times X$, defines a map from $P_X\to A_1(X)$, which is injective. By the fact that $A_1(X)$ is dominated by $\Pic^0(\wt{D_i})$, we have the subset of 

$$R=\{(a,b)|Z_*(a)=j_*(b)\}\subset P_X\times \Pic^0(\wt{D_i})$$

Here, $j$ is the regular map from $\wt{D_i}$ to $X$.

By the Mumford-Roitman argument we have $R=\cup_i R_i$ a countable union of algebraic subvarieties of the product: 

$$P_X\times \Pic^0(\wt{D_i})$$ 

This $R$ projects surjectively onto $P_X$, and hence there exists $R_i$, which surjects onto $P_X$, because we are working  over uncountable algebraically closed fields. So, we have 
$$R_i\subset P_X\times \Pic^0(\wt{D_i})$$

which can be viewed as a correspondence. Taking a smooth hyperplane section of $R_i$, we have $R_i$ maps finitely to $P_X$. Then given $a$ in $P_X$, there exist finitely many $b$ in $R_i$, such that 

$$Z_*(a)=j_*(b)$$ 

Denote $R_i$ by $R$ for simplicity.

Hence $$R_*(a)= b_1+b_2+\cdots+b_n$$ and so,

$$j_*(R_*(a))=\sum_i j_*(b_i)=nZ_*(a)\;.$$

The above $R_*$ is a homomorphism of abelian varieties as $R_*$ by abuse of notation is just the composition of 
$$R_*: P_X\to A_0(\Pic^0(\wt{D_i}))$$
and $$Alb: A_0(\Pic^0(\wt{D_i}))\to \Pic^0(\wt{D_i})\;.$$ 

Now, by the definition of $R_*$, it is a regular map of abelian varieties.

Composing with the Abel-Jacobi map, we have

$$AJ\circ Z_*=Id$$
 we have

 $$AJ(j_*(R_*(a)))=na$$

 If we have $R_*(a)=0$, then it follows that $na=0$ and hence $a$ is $n$ torsion for a fixed $n$. So, we have 
 
 $$R_*: P_X\to \Pic^0(\wt{D_i})$$ 
 
 with a finite kernel. Therefore, we have $Z_*$ with rational coefficient factors through $\Pic^0(\wt{D_i})$. Hence $P_X$ cannot be contained in 
 
 $$\ker(\Pic^0(\wt{D_i})\to A_1(X))$$ 
 
 Therefore, the map of $P_X^{\vee}\to P_X$ is not a zero map and is therefore an isogeny of some positive degree.

 Now, this is not yet, we can conclude that the map $\Pic^0(\wt{D_i})$ to $P_X$ is an isogeny. As before, we have some component $R'$ of $R$ that is surjectively map onto $\Pic^0(\wt{D_i})$. This is because we are working with complex numbers. By the simplicity of $P_X$, it follows that $R'$ finitely maps to $\Pic^0(\wt{D_i})$. Hence $R'$ is a correspondence on $P_X\times \Pic^0(\wt{D_i})$. By the universal definition of $P_X$, we have that $R'_*$ is a regular homomorphism of abelian varieties (the definition is the same as in the case of $R_*$).

 Observe here that if

 $$Z_*(a_1)=j_*(b)=Z_*(a_2)$$

 then $$a_1=a_2$$
 
 as $Z_*$ is injective. Therefore given a $b$, there exists a unique $a$ such that
 
 $$Z_*(a)=j_*(b)\;.$$
 
 Define $R'_*(b)=a$ and
 applying $Z_*$ we have 
 
 $$Z_*R'_*(b)= Z_*(a)=j_*(b)$$

 This tells us that if the kernel of $j_*$ is positive dimensional, we have $R'_*$ has a positive dimensional kernel.  Consider the composition of $R_*\circ R'_*(b)$ that is 
 
 $$R_*(a)=\sum_i b_{i}$$
 
 such that for each $b_i$ we have $$j_*(a)=Z_*(b_i)$$
 Put $R'_*(b)=a$.
 Now we have 
 
 $$j_*R_*R'_*(b)=\sum_i j_*(b_{i})=\sum_i Z_*(a)=nj_*(b)$$
 
 This means that 
 
 $$R_*R'_*(b)-n(b)$$
 
 is in the kernel $B$ of $j_*$ for all $b$ in $\Pic^0(\wt{D_i})$. Suppose that this $B$ is positive-dimensional.  

 So, we have

 $$AJj_*(R_*R'_*(b)-nb)=nR'_*(b)-AJj_*(nb)=0$$

 This means that

 $$n(R'_*(b)-AJj_*(b))=0$$

 This happens for all $b$. We have for $b=R_*(a)$

 $$n(R'_*(R_*(a))-AJj_*R_*(a))=0$$

 This implies that

 $$n(R'_*R_*(a)-na)=0$$

 For all $a$. This implies that

 $$R'_*R_*(a)=na\;.$$

 On the other hand, 
 
 $$Z_*(a)=j_*(b)$$
 
 shows that for a unique $b$, there is a unique $a$, such that the above occurs. Then considering $R'_*$, we have a map from $B\to P_X$, such that 
 
 $$R'_*(b)=a$$
 
 where $Z_*(a)=j_*(b)$
 
 Now
 
 $$R_*R'_*(b)=R_*(a)=\sum_i b_i$$
 
 such that 
 
 $$Z_*(a)=j_*(b_i)$$
 
 So we have,
 
 $$j_*R_*R'_*(b)=\sum_i j_*(b_i)=nZ_*(a)=nj_*(b)$$
 
 This gives us that 
 
 $$j_*(R_*R'_*(b)-nb)=0$$
 
 for $b$ in $\Pic^0(\wt{D_i})$.  Therefore, given $b$, there is a unique $a$, which in turn gives $b_i$, such that $b\mapsto \sum_i b_i$. The above tells us that 
 $$j_*(\sum_i b_i-nb)=0$$ 
 for all $b$.
 
So we have

$$j_*(\sum_i b_i)=j_*(nb)$$

Now each $b_i$ maps to $\sum_j b_{ij}$, so we have 

$$b\mapsto \sum_i b_i\mapsto \sum_{ij}b_{ij}$$
such that
$$j_*(\sum_{ij}b_{ij})=\sum_{ij}j_*(b_{ij})=\sum_i nj_*(b_{i})=j_*(n^2b)\;.$$
Since the above happens for all $b$, we can put $b=R_*(a)$. Then we have 
$$j_*(n^2R_*(a))=j_*R_*(n^2a)$$
Composing with $AJ$, we have 
$$AJ(j_*R_*(n^2a))=n^3a$$
This is equal to 
$$AJ(j_*(\sum_i b_{ij}))=AJ\circ Z_*(\sum_{ij}a)=n^2a$$
Hence we have the equality after applying $AJ$,
$$n^2a=n^3a$$
for all $a$ in $P_X$. Hence, all elements in $P_X$ cannot be torsion of order $n^3-n^2$, and therefore we have $$R_*R'_*=nId$$ unless $n=1$. But if $n=1$, then we have $R_*(a)=b$ such that $Z_*(a)=j_*(b)$. In that case, we have $$ R_* R'_* (b)=R_* (a)=b$$
with $Z_*(a)=j_*(b)$. So we have $R_*$ both injective and surjective.

Hence, we find that $\Pic^0(\wt{D_i})\to A_1(X)\cong P_X$
is finite and of degree $n\geq 1$. So, it means that $P_X$ isogenous to the Picard variety of a smooth projective surface. So, $\Pic^0(\wt{D_i})$ is a simple abelian variety of dimension $5$. Moreover, this abelian variety has the property that it is not isogenous to the Jacobian of a smooth projective curve. Now $\wt{D_i}$ is a surface of irregularity $5$. Taking into account a hyperplane section $C_t$ of $\wt{D_{i}}$, we have $J(C_t)$ surjecting onto $Alb(\wt{D_i})$, which is isogenous to $P_X$. So we have $J(C_t)\to P_X$ a surjective map.

As before, we consider the subset:

$$R:=\{(z,w)|Z_*(z)=j_*(w)\}\subset P_X\times J(C_t)$$

Given any $z$, we have that there exists $w$ such that 

$$Z_*(z)=j_*(w)$$

Hence $R$ maps surjectively onto $P_X$ and $R$ is a  countable union of Zariski closed subsets of $P_X\times J(C_t)$. It follows that there exists a component of $R$ say $R_i$ that maps dominantly onto $P_X$. We can assume by taking a hyperplane section of $R_i$, that $R_i\to P_X$ is finite. Then given $z$ in $P_X$ there exists $w_1,\cdots, w_n$ such that 

$$R_{i*}(z)=\sum_i w_i$$

Applying $j_*$ on both sides we have 

$$\sum_i j_*(w_i)=j_*(R_{i*})(z)=nZ_*(z)$$

Applying $AJ$ we have

$$nz=AJ(j_*R_{i*}(z))$$

If we have $R_{i*}(z)$ is zero, then it follows that $z$ is n-torsion on $P_X$. The above equality tells us that

$$AJ\circ j_*\circ R_{i*}=nId $$

On the other hand, we have $a=AJ(j_*(b))$, that is, for a $b$ in $J(C_t)$, there exists $a$ such that the above occurs. So, we have

$$AJ\circ Z_*(a)=AJ( j_*(b))$$

Cancelling $AJ$ from both sides, we have

$$Z_*(a)=j_*(b)$$

Moreover, this $a$ is unique to $b$ as $Z_*$ is injective. So we have $R'_*$ defined by 

$$R'_*(b)=a$$

Now 

$$R_{i*}R'_*(b)=R_{i*}(a)=\sum_i b_i$$

where $b_i$'s are such that $j_*(b_i)=Z_*(a)$.

Then we have 

$$j_*R_{i*}R'_*(b)-nj_*(b)=nZ_*(a)-nZ_*(a)=0$$

for all $b$ in $J(C_t)$. That is, we have: 

$$j_*(\sum_i b_i-nb)=0$$

Taking as before $b=R_{i*}(a)$ we have 

$$AJ\circ j_*\circ R_{i*}(a)=na$$

and $$j_*(\sum_i b_i)=nZ_*(a)\;.$$ 

Applying $AJ$ we have 

$$na=n^2a$$

This means that for all $a$ in $P_X$, it is $n^2-n$-torsion if $n$ is not equal to $1$. This is not possible unless $n=1$ but if $n=1$, then we have 
$$R_{i*}R'_*(b)=R_{i*}(a)=b$$
such that $Z_*(a)=j_*(b)$.
Hence $R_{i*}$ is surjective and injective. 
This means that there is an isogeny from $J(C_t)\to P_X$, by the existence of the universal codimension-two cycle $Z$, which is impossible by \cite{NP}.

\end{proof}

As a consequence of the previous result and the result by Voisin \cite{Voi1}[Theorem 1.7], that a smooth cubic threefold admits the Chow-theoretic decomposition of the diagonal if and only if the class of  $\theta^4/4!$ is algebraic on $P_X$, here $\theta$ is a divisor on $P_X$ giving the polarization on $P_X$. By our previous theorem it follows that:

\begin{corollary}
The class of $\theta^4/4!$ is not algebraic on $P_X$ for a very general $X$.
\end{corollary}

\subsection*{Acknowledgements} The author thanks SRM AP for hosting this project.

\end{document}